\documentclass[10pt]{amsart}
\usepackage{amssymb,amsmath}
\usepackage{mathtools}
\usepackage{graphicx}

\begin{document}

\newtheorem{thm}{Theorem}[section]
\newtheorem{lemma}[thm]{Lemma}
\newtheorem{defin}[thm]{Definition}
\newtheorem{rmk}[thm]{Remark}
\newtheorem{conj}[thm]{Conjecture}
\newtheorem{cor}[thm]{Corollary}
\newtheorem{mr}[thm]{Main Result}
\newtheorem{ass}[thm]{Assumption}
\newtheorem{prop}[thm]{Proposition}
\newtheorem{qu}[thm]{Question}

\makeatletter

\newcommand{\explain}[2]{\underset{\mathclap{\overset{\uparrow}{#2}}}{#1}}
\newcommand{\explainup}[2]{\overset{\mathclap{\underset{\downarrow}{#2}}}{#1}}

\makeatother

\newcommand{\map}{\mbox{$\rightarrow$}}
\newcommand{\bbb}{\mbox{$\beta$}}
\newcommand{\la}{\mbox{$\lambda$}}
\newcommand{\aaa}{\mbox{$\alpha$}}
\newcommand{\eee}{\mbox{$\epsilon$}}
\newcommand{\Rrr}{\mbox{$\mathbb{R}$}}
\newcommand{\lpd}{\mbox{$L^{\mathcal{V}(P,D^*)}$}}
\newcommand{\fpd}{\mbox{$\mathcal{V}(P,D^*)$}}
\newcommand{\bdd}{\mbox{$\partial$}}
\newcommand{\ra}{\rightarrow}
\newcommand{\Li}{\mbox{$L_+^{in}$}}
\newcommand{\Lo}{\mbox{$L_+^{out}$}}

\title{Bridge Number and Tangle products}
\author{Ryan Blair}
\thanks{Research partially supported by NSF and JSPS grants.}
\begin{abstract}
We show that essential punctured spheres in the complement of links with distance three bridge spheres have bounded complexity. We define the operation of tangle product, a generalization of both connected sum and Conway product. Finally, we use the bounded complexity of essential punctured spheres to show that the bridge number of a tangle product is at least the sum of the bridge numbers of the two factor links up to a constant error.
\end{abstract}
\maketitle

\section{Introduction}

Bridge number is a classical link invariant originally introduced by Schubert as a tool to study companion tori. In \cite{HSCH54}, Schubert proves the remarkable fact that given a composite knot $K$ with summands $K_{1}$ and $K_{2}$, the following equality holds:
$\beta(K)=\beta(K_{1})+\beta(K_{2})-1$, where $\beta(L)$ denotes the bridge number of a link $L$. Schultens gives a modern proof of this equality in \cite{JSCH01}.

Connected sum is a classical and intentionally restrictive method of amalgamating two links in $S^3$ together to create a new link in $S^3$. Tangle products are the natural generalization of this amalgamation operation. Roughly speaking, to form an $n$-strand tangle product of links $K_1$ and $K_2$ remove an $n$-strand rational tangle from the 3-sphere containing $K_1$ and the three sphere containing $K_2$. Now, glue the resulting tangles together via some homeomorphism of the $2n$-punctured sphere. The result is a tangle product, denoted $K_{1} \ast_{S} K_{2}$. For a rigorous definition see Section \ref{sec:tangleProd}. In particular, connected sums are 1-strand tangle products and Conway products are 2-strand tangle products. Conway products were studied in \cite{ST06} where Scharlemann and Tomova produced Conway products which respected multiple bridge surfaces. How bridge number behaves with respect to Conway products was studied in \cite{B}. The goal of this paper is to generalize Schubert's equality for bridge number to the operation of tangle product.

Because of their generality and the choices involved, tangle products are exceptionally poorly behaved. For example, for any two $n$ bridge knots $K_1$ and $K_2$ there exists an $n$-strand tangle product $K_{1} \ast_{S} K_{2}$ isotopic to the unknot. Hence, we will have to restrict our hypothesis to achieve a meaningful lower bound for $\beta(K_{1} \ast_{S} K_{2})$ in terms of $\beta(K_{1})$ and $\beta(K_2)$. Similarly, if $U$ is the unknot, there exist tangle products $U \ast_{S} U$ of arbitrarily high bridge number. This observation implies, in the absence of additional information, that there does not exist an upper bound for $\beta(K_{1} \ast_{S} K_{2})$ in terms of $\beta(K_{1})$ and $\beta(K_{2})$. The following is the main theorem.

\begin{thm}\label{main}
Given an $n$-strand tangle product $K_1\ast_{S}K_2$ such that there exists a minimal bridge sphere for $K_1\ast_{S}K_2$ of distance at least three and the product sphere $S$ is c-incompressible, then $\beta(K_1\ast_{S}K_2) \geq \beta(K_1) + \beta(K_2) - n(10n-6)$.
\end{thm}

In \cite{B}, the key additional hypothesis needed to produce a lower bound on the bridge number of a Conway product in terms of the bridge number of the two factor links was that bridge position and thin position coincide for the Conway product. In contrast, the result presented here is heavily dependent on the hypothesis that the tangle product has a minimal bridge sphere of distance at least three. This hypothesis allows for a much stronger structure theorem then that found in \cite{B} and, thus, a more restrictive lower bound on the bridge number of a tangle product. It is important to note that having the property that bridge position is thin position and having the property that a minimal bridge sphere is distance at least three are believed to be independent conditions. The following structure theorem is of independent interest and is proven in Section \ref{structure}, h-level and taut are defined in Section \ref{prelim}, and standard and adjacent are defined in Section \ref{structure}.

\begin{thm}\label{thm:struc}
Suppose there exists a bridge sphere $\Sigma$ for $K$ of distance three or greater and $S$ is taut with respect to an $h$-level embedding of $\Sigma$, then the following hold:

1)there do not exist standard saddles $\sigma$, $\rho$, and $\tau$ for $S$ such that $\sigma$ is adjacent to $\rho$ and $\rho$ is adjacent to $\tau$.

2)every outermost disk of the foliation of $S$ induced by the standard height function contains at least one point of $K\cap S$.
\end{thm}

Bachman and Schleimer showed that twice the genus plus the number of boundary components of of an essential surface serves as an upper bound for the distance of any bridge sphere for a knot \cite{BSCH}. In other words, a high distance bridge sphere forces a high \emph{intrinsic} complexity for any essential embedded surface in a knot complement. In contrast, Theorem \ref{thm:struc} implies that if $K$ has a bridge sphere of distance 3 or greater, then any essential punctured sphere can be isotoped so that the number of saddles in the induced foliation is bounded with respect to the number of punctures. Colloquially, if there exists a bridge sphere that is not low distance, then every essential punctured sphere has low \emph{extrinsic} complexity.

\section{Preliminaries}\label{prelim}

In this paper we study smooth links in $S^3$. Our central tool will be the standard height function on $S^3$, $h:S^3\ra [-1,1]$. The level surfaces of $h$ foliate $S^3$ into concentric 2-spheres and two exceptional points. The \textbf{bridge number} of a link is the minimal number of maxima of $h|_K$ over all isotopic morse embeddings of $K$.

A \textbf{tangle} is an ordered pair $(B,T)$ where $B$ is a 3-ball and $T\subset B$ is a properly embedded collection of arcs and loops. An \textbf{untangle} is a tangle $(B,T)$ such that $T$ is a collection of boundary parallel arcs. We say an embedded surface in $S^3$ is $k$-punctured if it meets $K$ transversely in $k$ points. A \textbf{bridge sphere} for a link $K$ is an $k$-punctured sphere decomposing $(S^3,K)$ into two trivial tangles $(H_1,T_1)$ and $(H_2,T_2)$.

\begin{defin}
A bridge sphere $\Sigma$ is \textbf{h-level} if there exists a regular value $r$ such that $\Sigma$ is isotopic to $h^{-1}(r)$.
\end{defin}

Given an embedded morse surface $S$ in $S^3$, let $\digamma_{S}$ be the singular foliation on $S$ induced by $h|_{S}$. A \textbf{saddle} is any leaf of this foliation homeomorphic to the wedge of two circles. By standard position, we can assume that all saddles of $\digamma_{S}$ are disjoint from $K$.

\begin{defin}\label{pairbridge}Let $K$ be a link embedded in $S^3$ and $S$ be a surface embedded in $S^3$ which meets $K$ transversely. We say that the pair $(K,S)$ is in bridge position with respect to the standard height function on $S^3$ if $h$ is a morse function when restricted to both $K$ and $S$ and there exist $a,b \epsilon [-1,1]$ such that:

\begin{enumerate}
\item All maxima of $K$ and all maxima of $S$ lie in $h^{-1}((b,1))$
\item All minima of $K$ and all minima of $S$ lie in $h^{-1}((-1,a))$
\item All saddles of $S$ and all intersection points $S\cap K$ lie in $h^{-1}((a,b))$
\end{enumerate}

\end{defin}

\begin{lemma}\label{pairbridgeiso}
For any embeddings of $K$ and $S$ in $S^3$ there is an isotopy of first $S$ and subsequently $K$ such that the resulting pair $(K,S)$ is in bridge position. Moreover, these isotopies fix $S$ and $K$ outside of a neighborhood of their maxima and minima, preserve the number of maxima of $h|K$, and the number of saddles of $\digamma_S$.
\end{lemma}

\begin{proof}
After a small isotopy, we can assume that $h|_{K}$ and $h|_S$ are morse functions and $S$ intersects $K$ transversely. Additionally, by general position, we can assume that both $K$ and $S$ are disjoint from both $h^{-1}(1)$ and $h^{-1}(-1)$. Let $b_o$ be the largest value among the heights of all saddles of $\digamma_S$ and all points of $S\cap K$. Let $a_o$ be the smallest value among the heights of all saddles of $\digamma_S$ and all points of $S\cap K$. Let $a = a_o - \frac{1}{2}(1+a_o)$ and $b = b_o + \frac{1}{2}(1-b_o)$. Let $M_1$ be the highest maximum of $S$ that lies below $h^{-1}(b)$. Let $\alpha$ be a monotone arc connecting $M_1$ to any point in $S^3$ above $h^{-1}(b)$. Since $M_1$ is the highest maximum below $h^{-1}(b)$, we can choose $\alpha$ so that the interior of $\alpha$ is disjoint from both $K$ and $S$. The portion of the boundary of a regular neighborhood of $\alpha$ lying above $M_1$ is a monotone disk $D$ such that $D$ is disjoint from $K$ and $S$ except in its boundary. Let $D_M$ be a regular neighborhood of $M_1$ in $S$.  $D$ together with $D_M$ co-bound a 3-ball whose intersection with $S$ is $D_M$. Isotope $D_M$ to $D$ across this 3-ball. This isotopy fixes $K$, is supported in a neighborhood of $M_1$ in $S$, and raises one maximum of $S$ above $h^{-1}(b)$. Repeat this process until all maxima of $S$ lie above $h^{-1}(b)$. Now that there are no maxima of $S$ between $h^{-1}(a)$ and $h^{-1}(b)$ we can similarly isotope all the maxima of ${K}$ above $h^{-1}(b)$ via an isotopy that fixes $S$ and is supported on a neighborhood of the maxima of $h_K$. Hence, we have achieved $(1)$ in the definition of $(K,S)$ bridge position while preserving the number of maxima of $K$ and the number of saddles of $S$. By a symmetric argument, we can also achieve $(2)$ in the definition of $(K,S)$ bridge position while preserving the number of maxima of $K$ and the number of saddles of $S$. After these isotopies, our choice of $a$ and $b$ guarantee that $(3)$ in the definition of $(K,S)$ bridge position is satisfied.
\end{proof}

\begin{defin}
A punctured surface $S$ is \textbf{taut} with respect to an $h$-level bridge sphere $\Sigma$ if $\digamma_{S}$ contains the fewest number of saddles subject to $(K,S)$ being in bridge position.
\end{defin}

This notion of taut is different then that found in \cite{B}. Specifically, a taut surface in \cite{B} is one that has a minimal number of saddles subject to $h_{K}$ having a minimal number of maxima. In this paper, a taut surface has a minimal number of saddles subject to a specified bridge sphere appearing as a level sphere of the height function.

\section{The Saddle Structure of n-Punctured Spheres}\label{structure}

In this section we use the notion of taut introduced in the previous section to develop constraints on $\digamma_S$ when $S$ is an embedded essential punctured sphere. For a more detailed discussion of the following definitions and their applications see \cite{B}.

Any given saddle $\sigma = s_{1}^{\sigma} \vee s_{2}^{\sigma}$, lies in a level sphere $S_{\sigma}=h^{-1}(h(\sigma))$. Let $D_{1}^{\sigma}$ be the closure of the component of $S_{\sigma}-s_{1}^{\sigma}$ that is disjoint from $s_{2}^{\sigma}$ and $D_{2}^{\sigma}$ be the closure of the component of $S_{\sigma}-s_{2}^{\sigma}$ that is disjoint from $s_{1}^{\sigma}$.

A subdisk $D$ in $\digamma_{S}$ is monotone if its boundary is entirely contained in a leaf of $\digamma_{S}$ and the interior of $D$ is disjoint from every saddle in $\digamma_{S}$. In practice, we will use the term subdisk in a slightly broader sense, allowing $\partial(D)$ to be immersed in $S$ (i.e. $\partial(D)$ is a saddle). We say a monotone disk is \textbf{outermost} if its boundary is $s_{i}^{\sigma}$ for some saddle $\sigma$ and label the disk $D_{\sigma}$. Similarly, if $s_{i}^{\sigma}$ bounds an outermost disk $D_{\sigma}$, we say $\sigma$ is an outermost saddle. It is usually the case that only one of $s_{1}^{\sigma}$ and $s_{2}^{\sigma}$ is the boundary of an outermost disk, so, our convention is to relabel so that $\partial(D_{\sigma}) = s_{1}^{\sigma}$. We say $\sigma$ is an \textbf{inessential saddle} if $\sigma$ is an outermost saddle and $D_{\sigma}$ is disjoint from $K$.

Suppose $\sigma$ is an outermost saddle. $S_{\sigma}$ cuts $S^{3}$ into two 3-balls. The one that contains $D_{\sigma}$ is again cut by $D_{\sigma}$ into two 3-balls $B_{\sigma}$ and $B'_{\sigma}$. We chose the labeling of $B_{\sigma}$ and $B'_{\sigma}$ so that $\partial(B_{\sigma})=D_{1}^{\sigma} \cup D_{\sigma}$. We say a saddle $\sigma$ is \textbf{standard} if there is a monotone disk $E_{\sigma}$ in $S$ such that $\partial(E_{\sigma}) = \sigma$ and $E_{\sigma}$ is disjoint from $K$.

By general position arguments, we can assume every saddle $\sigma$
in $\digamma_{S}$ has a bicollared neighborhood in $S$ that is
disjoint from $K$ and all other singular leaves of $\digamma_{S}$.  The boundary of this bicollared neighborhood
consists of three circles $c_{1}^{\sigma}$, $c_{2}^{\sigma}$, and $c_{3}^{\sigma}$ where
$c_{1}^{\sigma}$ and $c_{2}^{\sigma}$ are parallel to $s_{1}^{\sigma}$ and $s_{2}^{\sigma}$
respectively. We can assume $c_{1}^{\sigma}$, $c_{2}^{\sigma}$, and $c_{3}^{\sigma}$ are level
with respect to $h$ and that $c_{1}^{\sigma}$ and $c_{2}^{\sigma}$ lie in the same
level surface.

Figure 1 illustrates all of the terminology outlined above.

\begin{figure}[h]
\centering \scalebox{1.2}{\includegraphics{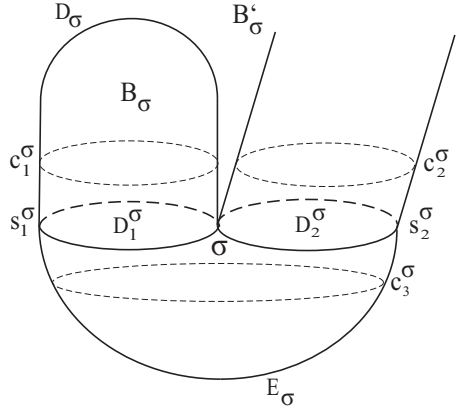}}
\caption{$\sigma$ is a standard, outermost saddle}\label{fig:labels1.eps}
\end{figure}

For the remainder of the paper $S$ will always denote a punctured sphere. We will denote the point $h^{-1}(1)$ as $+\infty$ and the point $h^{-1}(-1)$ as $-\infty$

\begin{lemma}\label{inftyball}Let $\sigma$ be an outermost saddle in $\digamma_{S}$.
 There is an ambient isotopy of $S$ that fixes $K$, lowers minima of $S$, raises maxima of $S$, and fixes $S$ outside of a neighborhood of the maxima and minima of $S$ such that, after this isotopy, $B_{\sigma}$ does not contain $+\infty$ or $-\infty$.
\end{lemma}

\begin{proof} See Lemma 1 \cite{JSCH01}. \end{proof}

The following lemma is an extension of Lemma 2 \cite{JSCH01} to our alternative notion of taut.

\begin{lemma}\label{iness}Suppose $\Sigma$ is an $h$-level bridge sphere for a link $K$. If $\digamma_S$ contains an inessential saddle, then $S$ is not taut with respect to $\Sigma$.
\end{lemma}

\begin{proof}
Suppose $(K,S)$ is in bridge position with respect to $\Sigma$, an $h$-level bridge sphere. Let $\sigma$ be an inessential saddle in $\digamma_S$. We can assume $D_{\sigma}$ contains a unique maximum and, by Lemma \ref{inftyball}, $B_{\sigma}$ does not contain $+\infty$. Let $(a,b)$ be the interval in the definition of $(K,S)$ bridge position. There exists an open interval $(p,q)$ such that $h(\sigma)\epsilon (p,q) \subset (a,b)$ and $h^{-1}(p,q)$ contains no saddles, maxima or minima of $\digamma_S$, no maxima or minima of $K$ and no points of $K\cap S$ other than $\sigma$. Let $s=h(\sigma)$. Horizontally shrink and vertically lower $B_{\sigma}$ so that the result of the isotopy, call it $B^*_{\sigma}$ is contained in $h^{-1}([s,q))$. This isotopy preserves the saddle structure of $\digamma_S$ and the number of critical values of $h_K$.

Let $S^{*}$ be the image of $S$ under this isotopy. Similarly, let $D^{*}_{\sigma}$ be the image of $D_{\sigma}$ under this isotopy and let $m$
be the unique maximum of $D^{*}_{\sigma}$. Let $M$ be the level
surface containing $m$. Hence, $M \cap S^{*}$ consists of
the point $m$ and a collection of circles.  One such circle $c_{2}$
is parallel in $S^{*}$ to $s_{2}$. Since the region between $M$ and $S_{\sigma}$ outside of $B^*_{\sigma}$ meets $S$ and $K$ in a collection of monotone annuli and vertical arcs respectively, we can choose a point $n$ in $c_{2}$ and an arc
$\alpha$ in $M$ that is disjoint from $S^{*}$ except at its boundary
$\{m,n\}$. Additionally, let $\beta$ be an arc in $S^{*}$ that does not meet
$K$, has boundary $\{m,n\}$ and is transverse to $\digamma_{C}$
everywhere accept where it passes through $s_{1} \cap s_{2}$ so that $\alpha$ and $\beta$ cobound a vertical disk $F$ that is disjoint from $K$ and disjoint from $S$ except
along $\beta$. Isotope $S^{*}$ along $F$ to effectively cancel a saddle
with a maximum. See Figure \ref{fig:thm1fig1.eps}. Let $S^{**}$ be the image of $S^{*}$ under this second isotopy. Notice that the ambient isotopies described fix all of $S^3$ below $h^{-1}(s)$, hence, $h^{-1}(s)$ remains a bridge sphere for $K$ isotopic to $\Sigma$. The only way for $(K^{**},S^{**})$ to fail to be in bridge position is for maxima of $S^{**}$ or $K$ to lie below a saddle of $S^{**}$ or a point of $K\cap S$. In this case use the isotopy from Lemma \ref{pairbridgeiso} to raise the maxima of $S^{**}$ and subsequently the maxima of $K$. Thus, we have produced an ambient isotopy of $S$ and $K$ which reduces the number of saddles of $S$ but preserves both $\Sigma$ as an $h$-level bridge sphere and $(K,S)$ in bridge position. Hence $S$ is not taut with respect to $\Sigma$.
\end{proof}

\begin{figure}[h]
\centering \scalebox{.75}{\includegraphics{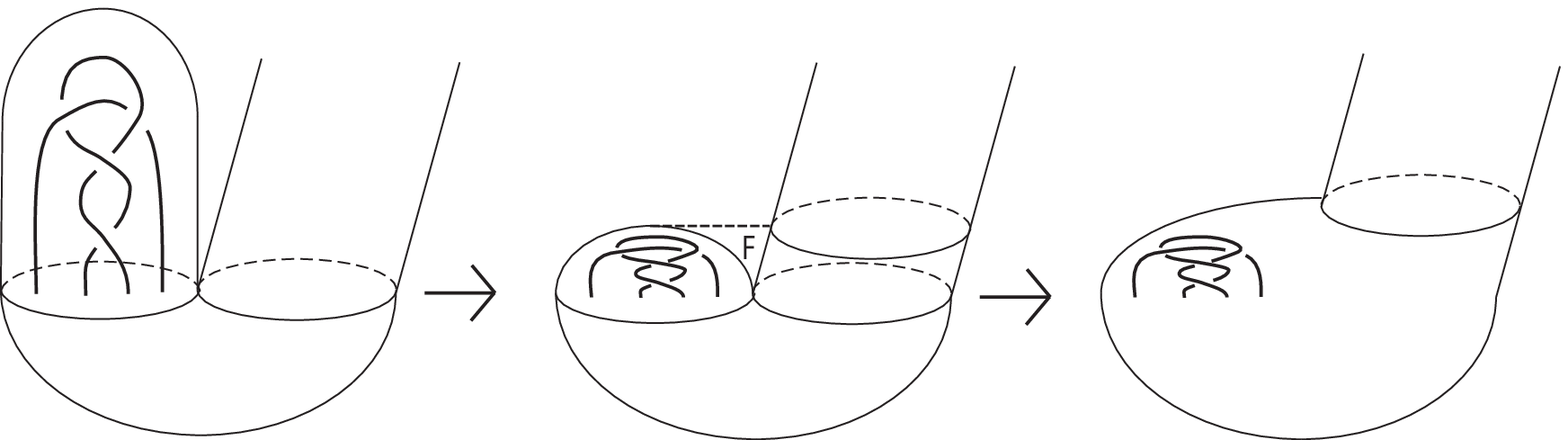}}
\caption{}\label{fig:thm1fig1.eps}
\end{figure}

\begin{defin}We say $\sigma$ is a \textbf{removable saddle} if
$\sigma$ is an outermost saddle where $D_{\sigma}$ has a unique maximum(minimum) and $h|_{K \cap B_{\sigma}}$ has a
local end-point maximum(minimum) at every point of $K \cap D_{\sigma}$. See Figure \ref{fig:removeable.eps}. Otherwise, we say $\sigma$ is non-removable. \end{defin}

\begin{figure}[h]
\centering \scalebox{1.2}{\includegraphics{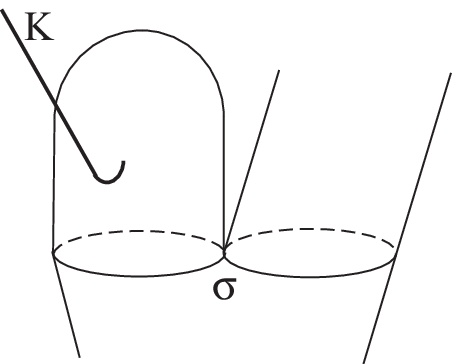}}
\caption{}\label{fig:removeable.eps}
\end{figure}

The following is an adaptation of Lemma 3 of \cite{B} to the notion of taut presented here.

\begin{lemma}\label{remove}Suppose $\Sigma$ is an $h$-level bridge sphere for a link $K$. If $\digamma_S$ contains a removable saddle, then $S$ is not taut with respect to $\Sigma$.\end{lemma}

\begin{proof}
Suppose $(K,S)$ is in bridge position with respect to $\Sigma$, an $h$-level bridge sphere. Let $\sigma$ be an removable saddle in $\digamma_S$. We can assume $D_{\sigma}$ contains a unique maximum and, by Lemma \ref{inftyball}, $B_{\sigma}$ does not contain $+\infty$. Applying the isotopy presented in Lemma \ref{iness}, we see that each point $x_i\epsilon K\cap D_{\sigma}$ together with the image of $x_i$ in $B^*_{\sigma}$ bound monotone subarcs of $K^*$. See Figure \ref{fig:tightcase2.eps}. Since $\sigma$ is removable neither the $x_i$ nor the $x^{*}_i$ is a maximum or minimum. Thus, we have eliminated a saddle of $S$ while preserving $(K,S)$ bridge position, a contradiction to tautness of $S$.

\end{proof}

\begin{figure}[h]
\centering \scalebox{.7}{\includegraphics{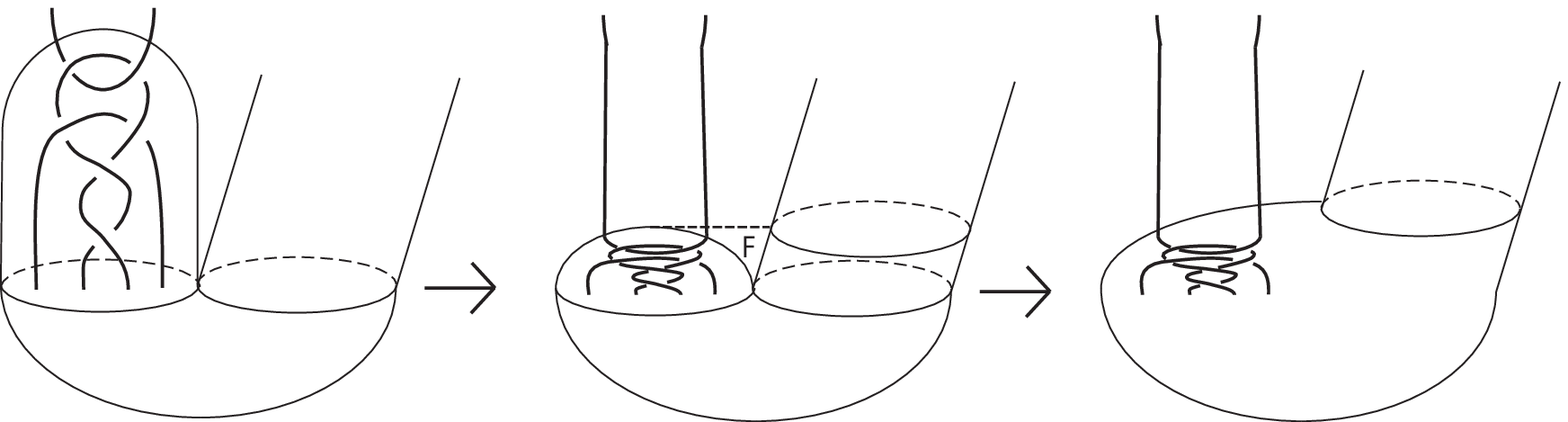}}
\caption{}\label{fig:tightcase2.eps}
\end{figure}

$S$ decomposes $S^{3}$ into two 3-balls $B_{1}$ and $B_{2}$. Let $\sigma$ be a saddle in $\digamma_{S}$ and
$L$ be the level sphere containing $c_{1}^{\sigma}$ and $c_{2}^{\sigma}$.
$L-(c_{1}^{\sigma} \cup c_{2}^{\sigma})$ is composed of two disks and an annulus $A$.
If a collar of $\partial(A)$ in $A$ is contained in $B_{1}$, then we
say $\sigma$ is \textbf{unnested} with respect to $B_{1}$. If not, we say
$\sigma$ is nested with respect to $B_{1}$. We define nested and
unnested with respect to $B_{2}$ similarly.  Note that nested with
respect to $B_{1}$ is the same as unnested with respect to $B_{2}$
and nested with respect to $B_{2}$ is unnested with respect to
$B_{1}$.

Two saddles $\sigma = s_{1}^{\sigma}
\vee s_{2}^{\sigma}$ and $\tau =s_{1}^{\tau} \vee s_{2}^{\tau}$ in $\digamma_{S}$ are
\textbf{adjacent} if, up to labeling, $s_{i}^{\sigma}$ and $s_{j}^{\tau}$
cobound a monotone annulus in $S$ that is disjoint from $K$.

The following lemma is an extension of Lemma 3 in \cite{JSCH01} to our alternative notion of taut.

\begin{lemma}\label{nesting}
Suppose $\Sigma$ is an $h$-level bridge sphere for a link $K$. If $\digamma_S$ contains adjacent saddles $\sigma$ and $\tau$ where $\sigma$ is a standard saddle and $\sigma$ and $\tau$ are nested with respect to
different 3-balls, then $S$ is not taut with respect to $\Sigma$.
\end{lemma}

\begin{proof}
Do to the symmetry of the argument we can assume that $E_{\sigma}$ has a unique maximum. Since $\sigma$ and $\tau$ are adjacent, then, up to relabeling, $s^{\sigma}_{1}$ and $s^{\tau}_{1}$ cobound a monotone annulus $A$ in $S$ that is disjoint from $K$. After a small isotopy of $K$ and $S$, we can assume that $K \cup S$ meets $D^{\sigma}_{2}$ in a collection of points and simple closed curves. After a small tilt of $D^{\sigma}_{2}$, $A\cup E_{\sigma} \cup D^{\sigma}_{2}$ is a monotone disk. Eliminate the saddle $\tau$ by applying the isotopy from the proof of Lemma \ref{iness} to the 3-ball cobounded by the monotone disk $A\cup E_{\sigma} \cup D^{\sigma}_{2}$ and the level disk $D^{\tau}_1$ that is disjoint from $D^{\tau}_{2}$. See Figure \ref{fig:nesting.eps}. As in the proof of Lemma \ref{remove}, we see that each point of $K \cap D^{\sigma}_{2}$ together with its image under the isotopy cobound monotone sub arcs of $K$. Similarly, each curve $S \cap D^{\sigma}_{2}$ together with its image under the isotopy cobounds monotone annulus in $S$. Thus, we have produced an ambient isotopy of $S$ and $K$ which reduces the number of saddles of $S$ but preserves both $\Sigma$ as an $h$-level bridge sphere and $(K,S)$ in bridge position. Hence, $S$ is not taut with respect to $\Sigma$.

\end{proof}

\begin{figure}[h]
\centering \scalebox{.7}{\includegraphics{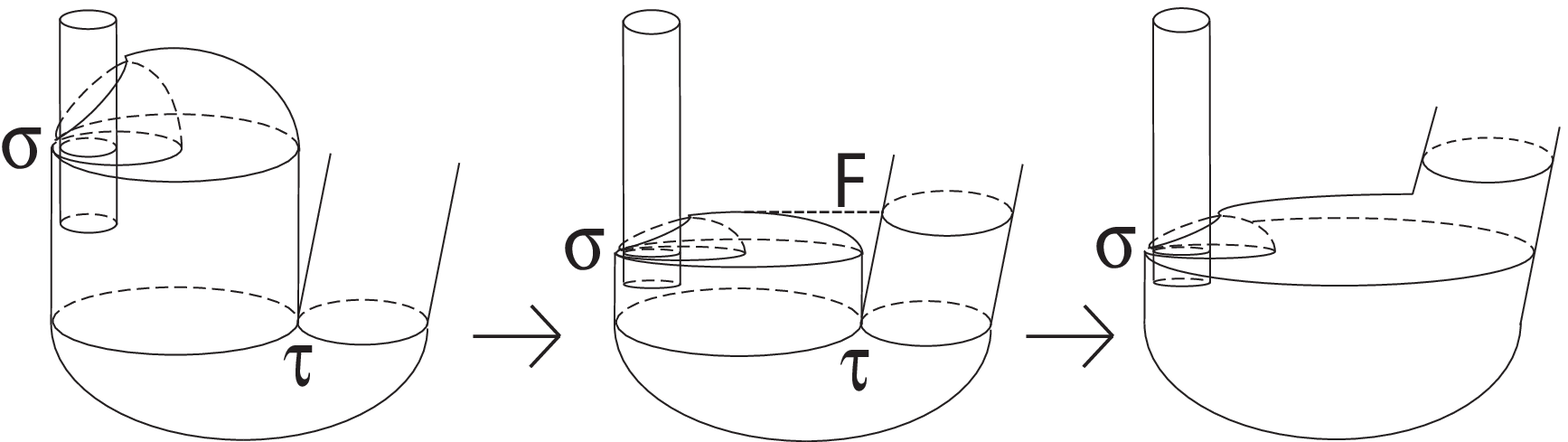}}
\caption{}\label{fig:nesting.eps}
\end{figure}

The previous lemmas in this section are independent of bridge sphere distance. Below we define the distance of a bridge sphere and obtain additional constraints on taut punctured spheres.

Let $\Sigma$ be a $2n$-punctured bridge sphere separating $(S^3,K)$ into two trivial $n$-strand tangles $(H_1,T_1)$ and $(H_2,T_2)$. Let
$\mathcal{C}_{n}$ be the curve complex for $\Sigma$. Let $\mathcal{V}_1$ be the set of all
isotopy classes of essential simple closed curves in $\partial(H_1)-T_1$ that bound disks in $H_1-T_1$. Define
$\mathcal{V}_2$ analogously. The \textbf{distance} of $\Sigma$, denoted $d(\Sigma)$, is the distance between $\mathcal{V}_1$ and $\mathcal{V}_2$ in $\mathcal{C}_{n}$ where the metric structure of $\mathcal{C}_{n}$ arises from assigning a length of one to each edge.

\begin{lemma}\label{prime}
Suppose $K$ is a link with a bridge sphere $\Sigma$ of distance 3 or greater, then $K$ is not split and $K$ is prime.

\end{lemma}

\begin{proof}
This follows immediately from Theorem 5.1 of \cite{BSCH}.

\end{proof}

\begin{lemma}\label{triple}
Suppose there exists a bridge sphere $\Sigma$ for $K$ of distance three or greater. If $S$ is c-incompressible punctured sphere taut with respect to an $h$-level embedding of $\Sigma$, then there do not exist standard saddles $\sigma$, $\rho$, and $\tau$ in $\digamma_S$ such that $\sigma$ is adjacent to $\rho$ and $\rho$ is adjacent to $\tau$.
\end{lemma}

\begin{proof}
By Lemma \ref{prime}, we know $K$ is non-split and prime. We will proceed by proving the contrapositive of the above statement. Suppose that three such saddles $\sigma$, $\rho$ and $\tau$ do exist. If $\digamma_S$ contains inessential or removable saddles, then $S$ is not taut by Lemma \ref{iness} or Lemma \ref{remove}. If $\sigma$ and $\rho$ are nested with respect to different three balls, then $S$ is not taut by Lemma \ref{nesting}. Similarly, if $\rho$ and $\tau$ are nested with respect to different three balls then $S$ is not taut by Lemma \ref{nesting}. Hence, we can assume that all three saddles $\sigma$, $\rho$ and $\tau$ have a common nesting and $\digamma_S$ contains no inessential and no removable saddles.

Assume that $E_{\rho}$ has a unique maximum and $h(\sigma)>h(\tau)$. If $E_{\rho}$ has a unique minimum or $h(\sigma)<h(\tau)$ the proof follows similarly.  By Lemma \ref{pairbridgeiso}, we can assume there is a partitioning of the critical values of $h|_K$ and $h|_S$ as in Definition \ref{pairbridge}. By standard Morse theory arguments, we can assume that the collection $\mathcal{C}$ of all critical points of $h|_K$, all critical points of  $h|_S$, and all points of $S\cap K$ occur at distinct heights. Choose $c$ so that $c$ is strictly between $h(\sigma)$ and the height of the next higher element of $\mathcal{C}$. Additionally, $h^{-1}(c)$ is a bridge sphere for $K$ that is isotopic to $\Sigma$. Although an abuse of notation, we will refer to $h^{-1}(c)$ as $\Sigma$. Recall the definition of $c^{\sigma}_1$ and $c^{\sigma}_2$ from Section 3 and Figure \ref{fig:labels1.eps}. We can assume that $c^{\sigma}_1$ and $c^{\sigma}_2$ lie on $\Sigma$.

Let $A_{\tau}$ be the monotone annulus with boundary $s^{\rho}_2 \cup s^{\tau}_1$. Since $\sigma$ is assumed to be higher than $\tau$ then $A_{\tau}$ intersects $B_K$ in a single simple closed curve $c_{\tau}$. In particular, $c_{\tau}$ is isotopic to both $s^{\rho}_2$ and $s^{\tau}_1$.  See Figure \ref{fig:triple.eps}.

\begin{figure}[h]
\centering \scalebox{.5}{\includegraphics{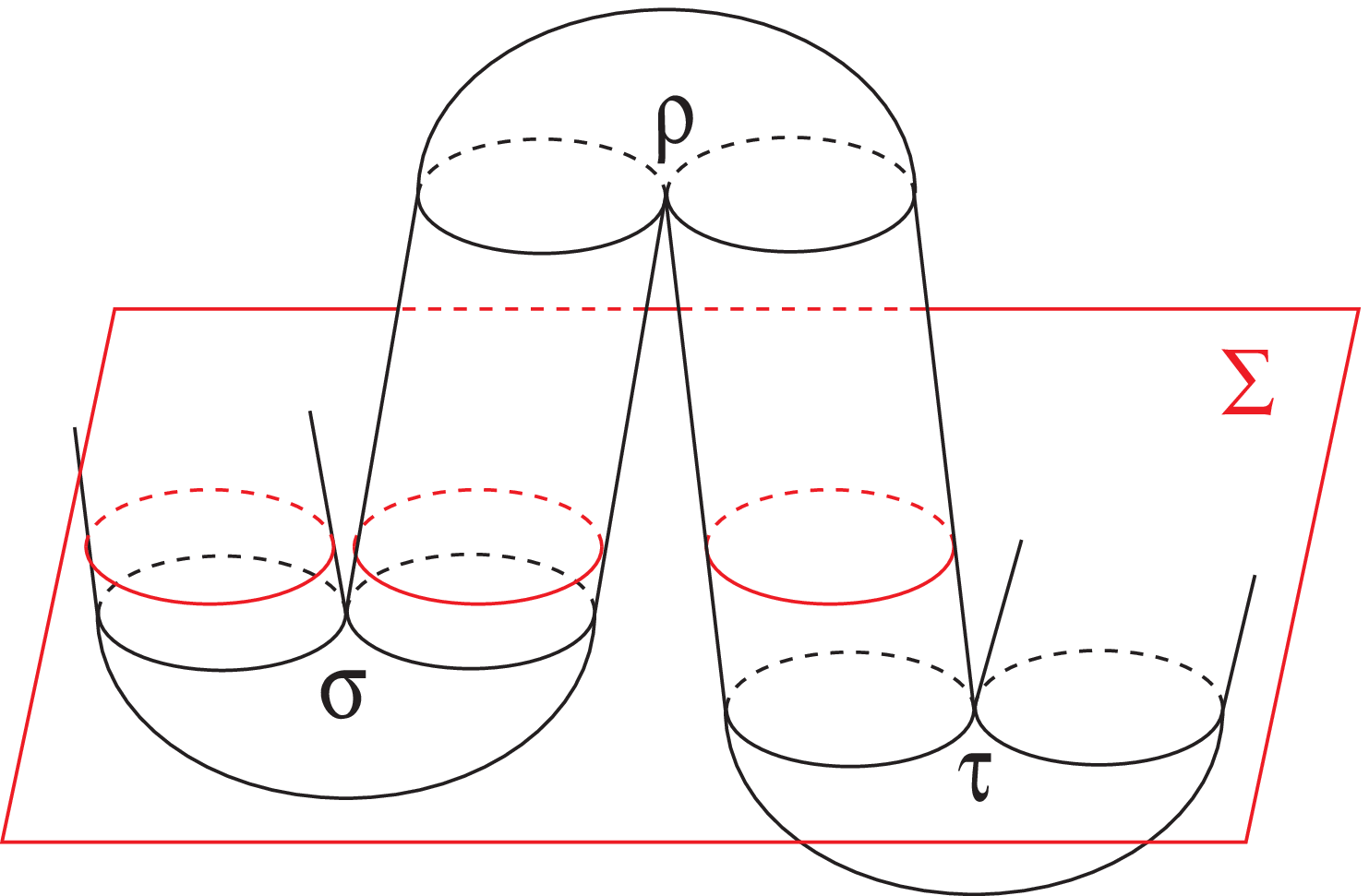}}
\caption{}\label{fig:triple.eps}
\end{figure}

Claim: If any of $c^{\sigma}_1$, $c^{\sigma}_2$ or $c_{\tau}$ bounds a zero or once punctured disk in $\Sigma$, then $S$ is not taut.

Proof of claim: Since all three saddles $\sigma$, $\rho$ and $\tau$ have a common nesting, each of $c^{\sigma}_1$, $c^{\sigma}_2$ and $c_{\tau}$ bound pairwise disjoint disks $E_1$, $E_2$ and $E_3$ respectively in $\Sigma$. To prove the claim it suffices to show that each of the disks $E_1$, $E_2$ and $E_3$ meet $K$ in at least two points.

Suppose, to form a contradiction, that $E_1$ is disjoint from $K$. $S \cap E_1$ is a collection of disjoint simple closed curves. An
innermost such curve $\gamma$ bounds a disk $D$ in $\Sigma$ that is disjoint from $K$ and meets $S$ only in its boundary. By c-incompressibility of $S$, $\gamma$ also bounds a disk $D'$ in $S$
that is disjoint from $K$. Since $K$ is non-split, $D$ and $D'$
cobound a 3-ball disjoint from $K$. Hence, we can eliminate $\gamma$ as an curve of intersection by isotopying $D'$ across this 3-ball and just past
$D$.  This isotopy leaves $K$ fixed and can only decrease the
number of saddles in $\digamma_{C}$. By repeating this process, we
can eliminate all curves of intersection of $S$ with
$int(E_1)$. Again, by c-incompressibility of $S$, $\partial(E_1)$
bounds a disk $D'$ in $C$ which is disjoint from $K$. Since $K$ is non-split $D'$ and $E_1$ cobound a 3-ball. Isotope $D'$ across this 3-ball to $E_1$ while fixing $K$. If $D'$ contains $\sigma$, then this isotopy eliminates $\sigma$, $\rho$ and $\tau$ while creating no new saddles. In this case $S$ was not taut. If $D'$ does not contain $\sigma$ then, after the isotopy, $\sigma$ is an inessential saddle. Hence, $S$ is not taut, by Lemma \ref{iness}.

Suppose, to from a contradiction, that $E_1$ meets $K$ in exactly one point. $S \cap E_1$ is a collection of disjoint simple closed curves. An
innermost such curve $\gamma$ bounds a disk $D$ in $\Sigma$ that  meets $S$ only in its boundary. If $D$ is disjoint from $K$ then apply the argument in the preceding paragraph to eliminate $\gamma$. Hence we can assume $D$ meets $K$ exactly once.  By c-incompressibility of $S$, $\gamma$ also bounds a disk $D'$ in $S$
that meets $K$ exactly once. Since $K$ is prime, $D$ and $D'$
cobound a 3-ball containing an unknotted arc of $K$. Hence, we can eliminate $\gamma$ as an curve of intersection isotopying $D'_K$ across this 3-ball and just past
$D_K$.  This isotopy leaves $K$ fixed and can only decrease the
number of saddles in $\digamma_{C}$. By repeating this process, we
can eliminate all curves of intersection of $S$ with
$int(E_1)$. Again, by c-incompressibility of $S$, $\partial(E_1)$
bounds a disk $D'$ in $C$ which meets $K$ exactly once and $E_1 \cup D$ is the boundary of a three ball containing an unknotted arc of $K$. Isotope $D'$ across this 3-ball to $E_1$ while fixing $K$. If $D'$ contains $\sigma$, then this isotopy eliminates $\sigma$, $\rho$ and $\tau$ while creating no new saddles. In this case $S$ was not taut. If $D'$ does not contain $\sigma$ then, after the isotopy, $\sigma$ is a removable saddle. Hence, $S$ is not taut, by Lemma \ref{remove}.

By applying nearly identical arguments, we can show that both $E_2$ and $E_3$ meets $K$ in at least two points. The claim then follows.$\square$

Let $M$ be the three ball above $\Sigma$ and $N$ be the three ball below $\Sigma$. By construction $c^{\sigma}_1$ and $c^{\sigma}_2$ cobound an annulus $A_{\sigma}$ properly embedded in $N$ and disjoint from $K$. Similarly, $c^{\sigma}_2$ and $c_{\tau}$ cobound and annulus $A_{\rho}$ properly embedded in $M$ and disjoint from $K$. Since $(M,K\cap M)$ and $(N,K \cap N)$ are both untangles then $A_{\sigma}$ and $A_{\rho}$ are boundary compressible in $M-K$ and $N-K$. Let $H_{\sigma}$ be the disk in $M$ gotten by boundary compressing $A_{\sigma}$ and $H_{\rho}$ be the disk in $M$ gotten by boundary compressing $A_{\rho}$. Both $\partial(H_{\sigma})$ and $\partial(H_{\rho})$ are disjoint from $c^\sigma_2$. By the above claim, $\partial(H_{\rho})$ and $\partial(H_{\sigma})$ are essential in $\Sigma$. Since $\partial(H_{\sigma}) \epsilon \mathcal{A}$, $\partial(H_{\rho}) \epsilon \mathcal{B}$, and both $\partial(H_{\sigma})$ and $\partial(H_{\rho})$ are disjoint from $c^\sigma_2$, $d(B)=d(\mathcal{A}, \mathcal{B})<3$.
\end{proof}

We summarize the results of the previous lemmas using the following definition.

\begin{defin}\label{admissdef}
A singular foliation $\digamma_S$ for a closed surface $S$ with $k$-marked points is \textbf{admissible} if it is induced by the standard height function on $S^3$ via some Morse embedding of $S$ into $S^3$ such that the following hold:

1)there do not exist standard saddles $\sigma$, $\rho$, and $\tau$ in $\digamma_S$ such that $\sigma$ is adjacent to $\rho$ and $\rho$ is adjacent to $\tau$.

2)every outermost disk of $\digamma_{S}$ contains at least one marked point.
\end{defin}

\textbf{Proof of Theorem \ref{thm:struc}}

\begin{proof}
By Lemma \ref{iness} and Lemma \ref{triple}, $\digamma_S$ is admissible when we view $S$ as a surface with marked points $S\cap K$.
\end{proof}

For a fixed surface type and number of puncture points, the number of saddles in any admissible singular foliation is bounded.

\begin{lemma}\label{numberofsaddles}
If $S$ is topologically a 2-sphere with $k$-marked points and $\digamma_S$ is admissible, then the number of saddles in $\digamma_S$ is at most $5k-8$.
\end{lemma}

\begin{proof}Let $\digamma_S$ be an admissible singular foliation for $S$ that has a maximal the number of saddles.

Claim: All marked points are contained in the collection of outermost disks of $\digamma_S$.

Proof of claim: Suppose not. Hence, there is a puncture point $x$ not on an outermost disk of $\digamma_S$. There is a non-singular leaf $\alpha$ in $\digamma_S$ containing $x$. Isotope $S$ in a neighborhood of $\alpha$ as in Figure \ref{fig:admiss1.eps}. The resulting singular foliation remains admissible but has one more saddle then $\digamma_S$, contradicting the maximality of the number of saddles of $\digamma_S$. $\square$

\begin{figure}[h]
\centering \scalebox{.4}{\includegraphics{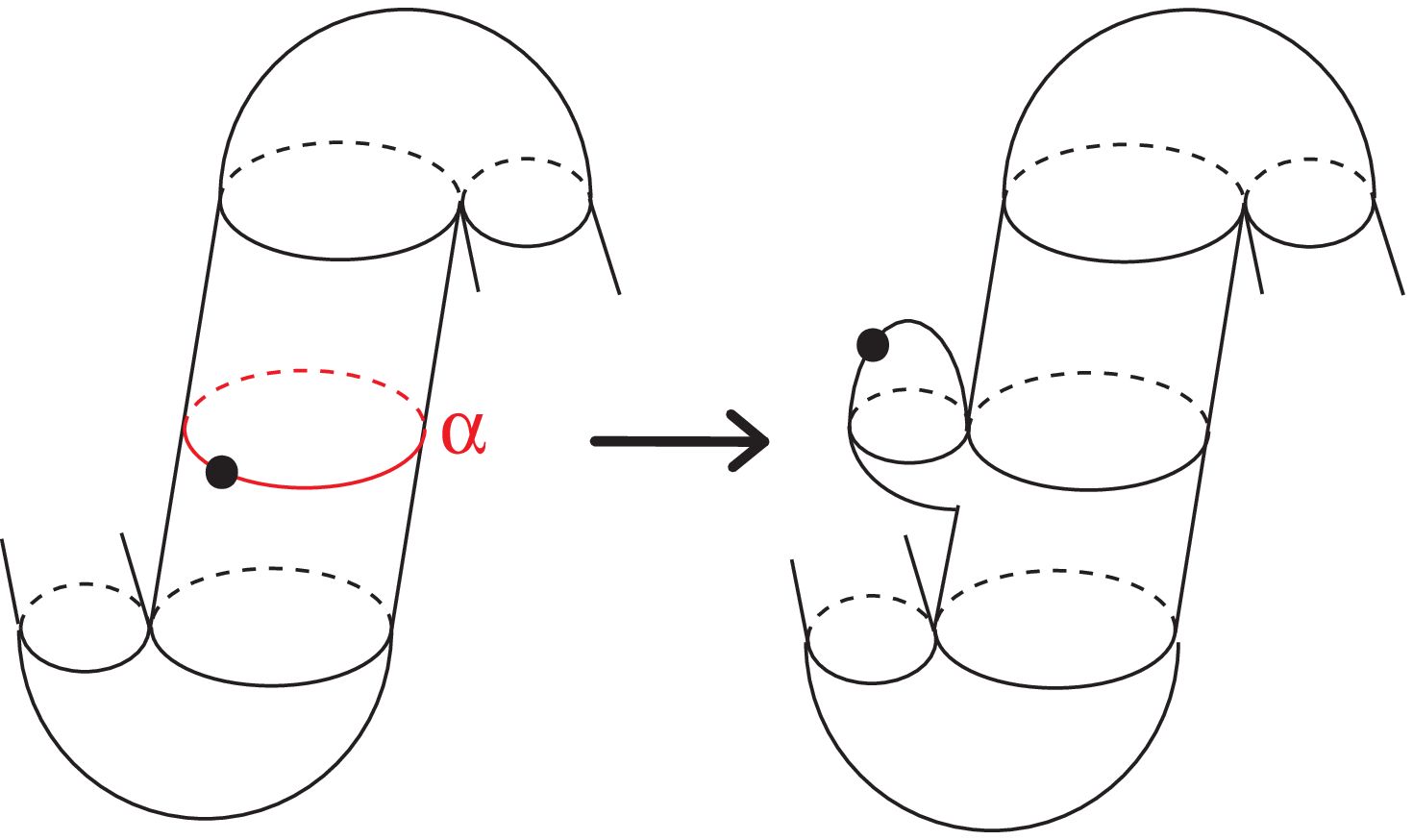}}
\caption{}\label{fig:admiss1.eps}
\end{figure}

Claim: All outermost saddles of $\digamma_S$ contain exactly one puncture point.

Proof of claim: Suppose not. Hence, there is an outermost disk $D_{\sigma}$ that contains two puncture points $x$ and $y$. Let $\alpha_x$ and $\alpha_y$ be the closed curves in $\digamma_S$ containing $x$ and $y$ respectively. Up to relabeling we can assume that $\alpha_y$ bounds a monotone disk in $\digamma_S$ that contains $x$. Alter $\digamma_S$ in a neighborhood of $\alpha_y$ as in Figure \ref{fig:admiss2.eps}. The resulting singular foliation remains admissible but has one more saddle then $\digamma_S$, contradicting the maximality of the number of saddles of $\digamma_S$. $\square$

\begin{figure}[h]
\centering \scalebox{.4}{\includegraphics{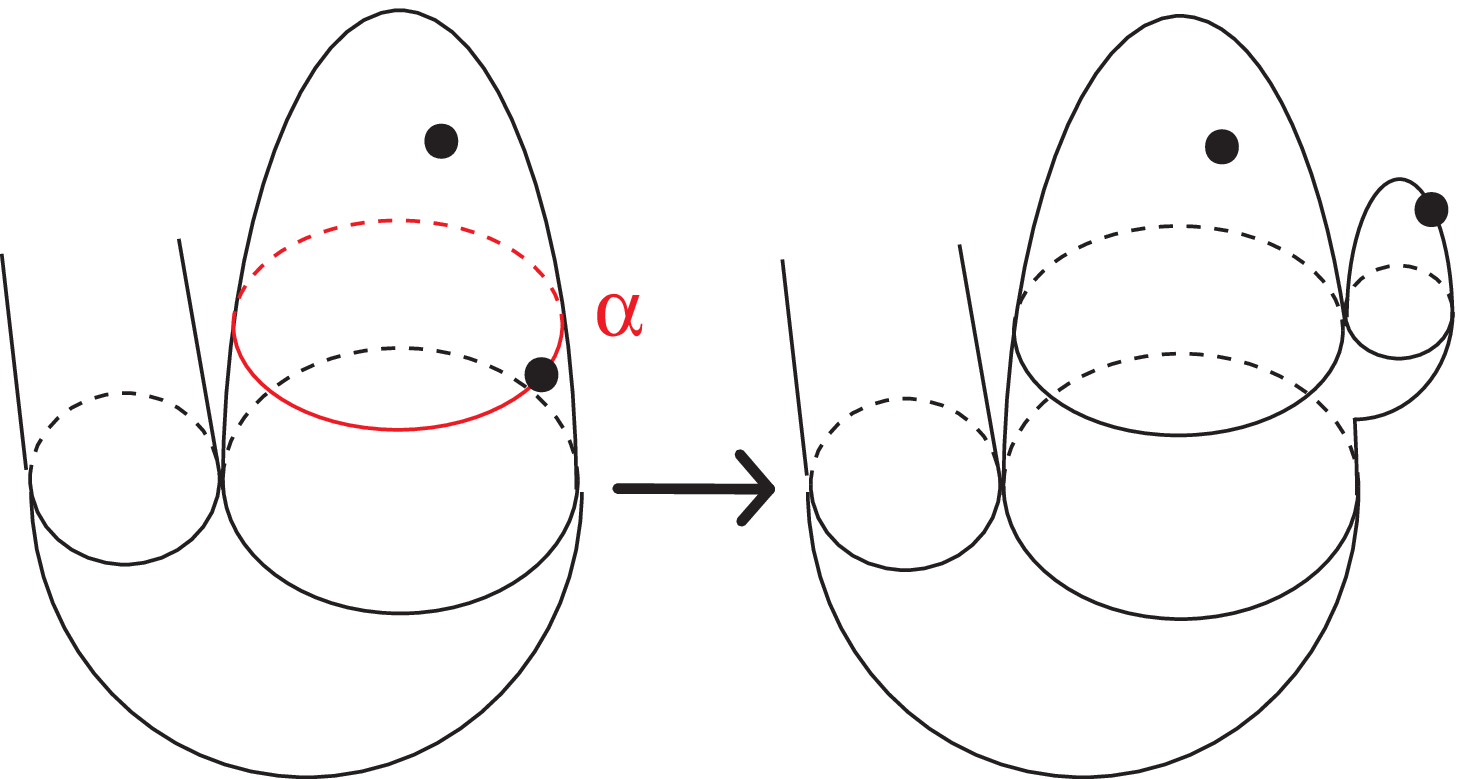}}
\caption{}\label{fig:admiss2.eps}
\end{figure}

We proceed by induction on $k\geq 2$.

Suppose $k=2$, then the maximal number of saddles is two, as depicted in Figure \ref{fig:admiss3.eps}.

\begin{figure}[h]
\centering \scalebox{.4}{\includegraphics{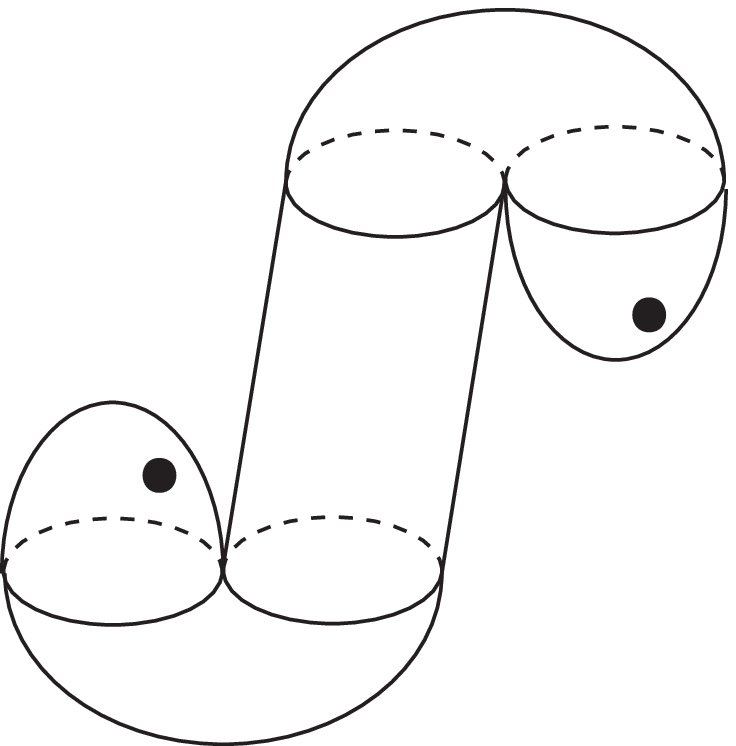}}
\caption{}\label{fig:admiss3.eps}
\end{figure}

Assume $\digamma_{S}$ has at most $5(l-1)-8$ saddles for $k=l-1$. We will show $\digamma_{S}$ has at most $5l-8$ saddles for $k=l$.

Choose $\digamma_S$ to have the maximal number of saddles of any admissible singular foliation. Since all puncture points are contained in outermost disks of $\digamma_S$ and all outermost disks of $\digamma_S$ contain exactly one puncture point, then $\digamma_{S}$ contains exactly $l-2$ non-standard saddles. Since we can assume $l>2$ there exists a non-standard saddle for $\digamma_{S}$. Label this saddle $\sigma$. By passing to an outermost non-standard saddle and possibly relabeling $c^{\sigma}_1$, $c^{\sigma}_2$ and $c^{\sigma}_3$ we can assume that both $c^{\sigma}_1$ and $c^{\sigma}_2$ bound disks $F^{\sigma}_1$ and $F^{\sigma}_2$ respectively in $\digamma_S$ such that both $F^{\sigma}_1$ and $F^{\sigma}_2$ are disjoint from all non-standard saddles. Each of $F^{\sigma}_1$ and $F^{\sigma}_2$ meets a unique outermost disk of $\digamma_S$. By condition $2)$ of the definition of admissible and the previous claim, each of these disks contains a unique marked point. By condition $1)$ of the definition of admissible, each of $F^{\sigma}_1$ and $F^{\sigma}_2$ contains exactly two saddles, all of which are standard. Replace a neighborhood of the union of $F^{\sigma}_1$ and $F^{\sigma}_2$ in $\digamma_S$ with a single monotone disk $M$ containing a unique marked point. Call the resulting singular foliation $\digamma^*_S$ and notice that $\digamma^*_S$ is a foliation for a sphere with $l-1$ marked points. Notice that $\digamma^*_S$ has 5 fewer saddles than $\digamma_S$. If we can show that $\digamma^*_S$ is admissible, then, by the induction hypothesis, $\digamma_S$ has at most $5(l-1)-8+5=5k-8$ saddles and we have proven the theorem.

Beginning with an embedding of a sphere with $l$ marked points in $S^3$ realizing $F_S$, use the isotopy from Lemma \ref{iness} to eliminate the outermost saddles in $F^{\sigma}_1$ and $F^{\sigma}_2$ and iterate this process until $\sigma$ is outermost and can be eliminated similarly. Since the isotopy in Lemma \ref{iness} preserves the singular foliation induced by the height function outside a neighborhood of $F^{\sigma}_1 \cup F^{\sigma}_2$, then this process produces an embedding of a sphere with $l-1$ marked points into $S^3$ with induced singular foliation $\digamma^*_S$.

If $\digamma^*_S$ fails Definition \ref{admissdef} criteria $1)$, then, by inclusion of the complement of $M$ in $\digamma^*_S$ into $\digamma_S$, $\digamma_S$ also fails criteria $1)$, a contradiction.

Let $D_{\tau}$ be an outermost disk of $\digamma^*_S$. If $D_{\tau}$ is again an outermost disk in $\digamma_{S}$ via inclusion, then $D_{\tau}$ contains one marked point by the admissibility of $\digamma_S$. If $D_{\tau}$ is not an outermost disk for $\digamma_{S}$, then $D_{\tau}$ contains $M$ and, thus, contains a marked point.

Hence, $\digamma_{S}^*$ is admissible.
\end{proof}

\section{Tangle Products}\label{sec:tangleProd}

In this section we define tangle product and use the combinatorial result of the previous section to show that, under suitable hypothesis, the bridge number of a tangle product is supper-additive up to constant error.

\begin{defin} A graph $G$ is an $n$-star graph if $G$ has $n$ edges and $n+1$ vertices such that $n$ of the vertices are valence one and one of the vertices is valence $n$. Denote by $\partial(G)$ the set of valence one vertices.
\end{defin}

\begin{defin}
Let $K_1$ and $K_2$ be links embedded in distinct copies of $S^3$, $S^3_1$ and $S^3_2$. Let $G_1$ and $G_2$ be $n$-star graphs embedded in $S^3_1$ and $S^3_2$ respectively such that $G_i\cap K_i=\partial(G_i)$. Let $\mu(G_i)$ be a small, closed, regular neighborhood of $G_i$ in $S^3_i$ such that $(\mu(G_i), K_i \cap \mu(G_i))$ is a rational tangle. Let $B_{i}=S^3_i-int(\mu(G_i))$. A link in $S^3$ obtained by gluing $\partial(B_1)$ to $\partial(B_2)$ via a homeomorphism such that points in $\partial(B_1)\cap K_1$ are mapped to points in $\partial(B_2)\cap K_2$ is called an \textbf{$n$-strand tangle product} of $K_1$ and $K_2$ and is denoted by $K_1\ast_{S}K_2$. The image of $\partial(B_1)$ and $\partial(B_2)$ under this identification is called the product sphere and is denoted $S$.
\end{defin}

\textbf{Proof of Theorem \ref{main}}

\begin{proof}
Let $\Sigma$ be a minimal bridge sphere for $K_1\ast_{S}K_2$ of distance at least three. We can assume that $\Sigma$ is $h$-level and that $S$ is taut with respect to $\Sigma$. By Theorem \ref{thm:struc}, $\digamma_S$ is admissible. By Lemma \ref{numberofsaddles}, $\digamma_S$ contains at most $10n-8$ saddles.

If the set of saddles $\digamma_S$ is nonempty, then it contains at least two outermost disks. Hence, there exists an outermost saddle $\sigma$ such that $D_{\sigma}$ meets $K_1\ast_{S}K_2$ in at most $n$ points. We can eliminate $\sigma$ via the ambient isotopy that horizontally shrinks and vertically lowers $B_{\sigma}$ as in the proof of Lemma \ref{iness}. This isotopy produces at most $n$ new maxima for $h_{K_1\ast_{S}K_2}$. See Figure \ref{fig:createmax.eps}. By repeating this process, we can eliminate all saddles of $\digamma_S$ at the cost of creating at most $n$ new maxima per saddle. Thus, we can assume we have an embedding of $K_1\ast_{S}K_2$ with at most $\beta(K_1\ast_{S}K_2)+n(10n-8)$ maxima such that $\digamma_S$ contains no saddles. Denote this embedding of $K_1\ast_{S}K_2$ by $K^*$.

\begin{figure}[h]
\centering \scalebox{.5}{\includegraphics{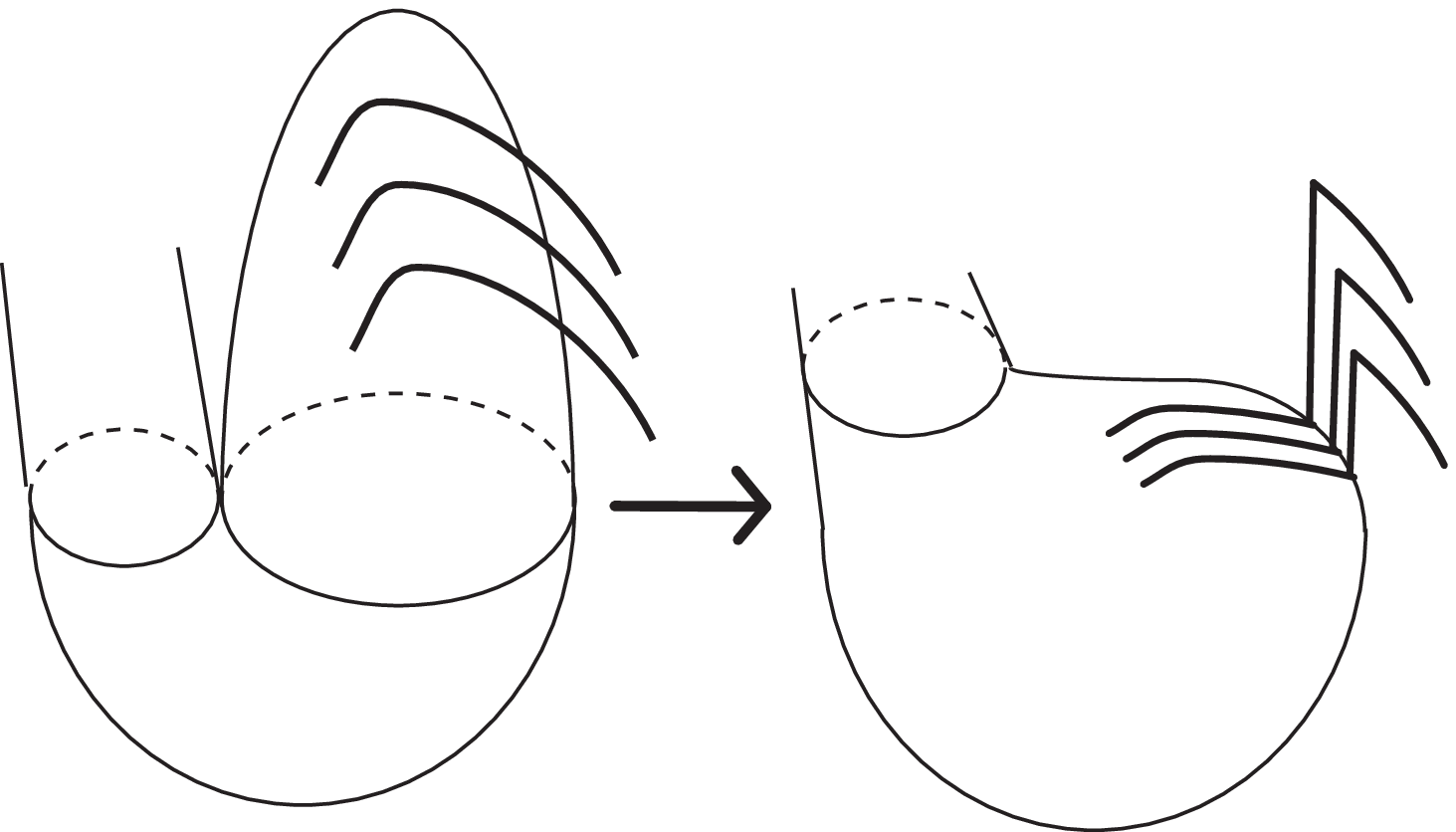}}
\caption{}\label{fig:createmax.eps}
\end{figure}

If $\digamma_S$ contains no saddles, there is a level preserving isotopy of
$S^{3}$ taking $S$ to a standard round 2-sphere. Such an isotopy preserves
the number of maxima of $h_{K^*}$. Recall that $S$ decomposes $S^3$ into two 3-balls $B_1$ and $B_2$. The link $K_{1}$ can be
recovered from the tangle $(B_1,K^* \cap B_{1})$ by glueing a trivial tangle $(B^3,R)$ to $(B_1,K^* \cap B_{1})$ along their common $2n$-punctured sphere boundary. See Figure \ref{fig:addingrational2.eps}. The number
of maxima of the resulting embedding of $K_{1}$ is at most $n$ more
than the number of maxima of $h|_{K^*}$ in $B_{1}$. By a similar argument, we can produce an embedding of $K_{2}$ with at most $n$ more
maxima than the number of maxima of $h|_{K^*}$ in $B_{2}$. Hence, $\beta(K_{1})-n+\beta(K_2)-n\leq \beta(K_1\ast_{S}K_2)+n(10n-8)$, or $\beta(K_1\ast_{S}K_2) \geq \beta(K_1) + \beta(K_2) - n(10n-6)$.

\begin{figure}[h]
\centering \scalebox{.5}{\includegraphics{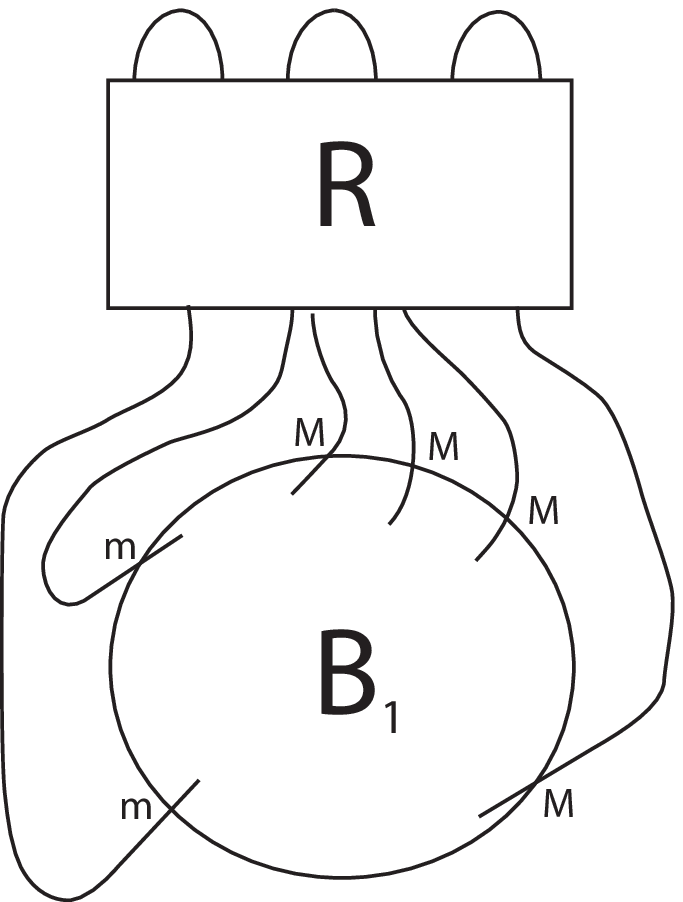}}
\caption{}\label{fig:addingrational2.eps}
\end{figure}

\end{proof}

\begin{rmk}
With more detailed analysis the constant $- n(10n-6)$ that appears in the statement of Theorem \ref{main} can be improved. However, the author believes it can not be improved beyond a quadratic expression in $n$ using the techniques presented in this paper.
\end{rmk}


\begin{thebibliography}{1}

\bibitem{BSCH} D. Bachman, S. Schleimer \newblock{Distance and Bridge Position} \newblock{\em Pacific Journal of Mathematics} 219:2 (2005), 221--235.

\bibitem{B} Ryan Blair \newblock{Bridge Number and Conway Products} \newblock{\em Algebr. Geom. Topol.} 10:789 --
    823 (electronic), 2010.

\bibitem{ST06} M. Scharlemann, M. Tomova, \newblock{Conway Products and Links
with Multiple Bridge Spheres} \newblock{Michigan Mathematics Journal} 56 (2008), no. 1, 113--144.

\bibitem{HSCH54} H. Schubert, \newblock{\"Uber eine Numerishe
Knoteninvariante,} \newblock{\emph{Math. Z.}} 61 (1954), 245--288.

\bibitem{JSCH01} J. Schultens, \newblock{Additivity of Bridge Number of
Knots} \newblock{\emph{Math. Proc. Cambridge Philos. Soc.}} 135 (2003),
539--544.

\end{thebibliography}
\end{document}